\documentclass{article}
\usepackage{arxiv}

\usepackage[utf8]{inputenc} 
\usepackage[T1]{fontenc}    
\usepackage{hyperref}       
\usepackage{url}            
\usepackage{booktabs}       
\usepackage{nicefrac}       
\usepackage[english]{babel}
\usepackage{ae}
\usepackage{eucal}
\usepackage[dvips]{graphicx}
\usepackage{epsfig}
\usepackage{graphicx}
\usepackage{amssymb}
\usepackage{amsthm}
\usepackage{amsfonts,amsmath}
\usepackage{amssymb}
\usepackage{a4}
\usepackage{apu}
\usepackage{xcolor}
\usepackage{amsmath}
\usepackage{mathtools}
\usepackage{multicol}
\usepackage[all]{xy}

\usepackage{dcolumn,amsthm}

\renewenvironment{proof}[1][Proof]{\noindent\textit{#1. } }{\hfill$\square$}

\newtheoremstyle{theorem}{6pt}{6pt}{\rm}{}{\sffamily}{ }{ }{}
\theoremstyle{theorem}
\newtheorem{theorem}{\sc Theorem}[section]

\newtheorem{lemma}[theorem]{\sc Lemma}

\newtheoremstyle{example}{6pt}{6pt}{\rm}{}{\sffamily}{ }{ }{}
\theoremstyle{example}
\newtheorem{example}{\sc Example}[section]

\newtheoremstyle{corollary}{6pt}{6pt}{\rm}{}{\sffamily}{ }{ }{}
\theoremstyle{corollary}
\newtheorem{corollary}{\sc Corollary}[section]

\newtheoremstyle{definition}{6pt}{6pt}{\rm}{}{\sffamily}{ }{ }{}
\theoremstyle{definition}
\newtheorem{definition}{\sc Definition}[section]

\newtheorem{problem}[theorem]{Problem}

\newtheoremstyle{remark}{6pt}{6pt}{\rm}{}{\sffamily}{ }{ }{}
\theoremstyle{remark}

\newtheoremstyle{approximation}{6pt}{6pt}{\rm}{}{\sffamily}{ }{ }{}
\theoremstyle{approximation}

\newtheoremstyle{scheme}{6pt}{6pt}{\rm}{}{\sffamily}{ }{ }{}
\theoremstyle{scheme}

\usepackage{biblatex}
\usepackage{csquotes}

\title{On some classes of Riemannian manifolds}

\author{
 Maryam Samavaki \\
  Department of Physics and Mathematics\\
  University of Eastern  Finland\\
  P.O. Box 111, FI-80101 Joensuu,Finland  \\
  \texttt{maryam.samavaki@uef.fi} \\
  \AND
  Jukka Tuomela \\
  Department of Physics and Mathematics\\
  University of Eastern  Finland\\
  P.O. Box 111, FI-80101 Joensuu,Finland  \\
  \texttt{jukka.tuomela@uef.fi} \\
}

\begin{document}
\maketitle

\begin{abstract}
We study several classes of Riemannian manifolds which are defined by imposing a certain condition on the Ricci tensor. We consider the following cases: Ricci recurrent, Cotton, quasi Einstein and pseudo Ricci symmetric condition. Such conditions can be interpreted as overdetermined PDE systems whose unknowns are the components of the Riemannian metric, and perhaps in addition some auxiliary functions. Hence even if the dimension of the manifold is small it is not easy to compute interesting examples by hand, and indeed very few examples appear in the literature. We will present large families of nontrivial examples of such manifolds. The relevant PDE systems are first transformed to an involutive form. After that in many cases one can actually solve the resulting system explicitly. However, the involutive form itself already gives a lot of information about the possible solutions to the given problem. We will also discuss some relationships between the relevant classes. 
\end{abstract}



\section{Introduction}
\label{intro}

In the following we are going to analyze and present examples of several classes of Riemannian manifolds: Ricci recurrent, pseudo Ricci symmetric, Cotton and quasi Einstein manifolds. Sometimes these classes are understood in a generalized sense without requiring the positive definiteness of the metric. In the following we will on the other hand only consider the positive definite case. However, in the actual construction of examples the positive definiteness does not play a role, so that using the same approach one can also construct examples which are not positive definite.

Ricci recurrent manifolds were first considered in \cite{patterson} (even earlier recurrent Riemannian manifolds were introduced in \cite{ruse}). Since then this class has been analyzed by many authors, see for example \cite{dede} and referneces therein, where also various generalizations and extensions of this notion are considered.

The concept of  pseudo Ricci symmetric manifold was perhaps explicitly first  introduced in \cite{Chaki1}. However, earlier in \cite{Chaki3} authors had obtained the characterizing condition when analyzing the existence of another structure on Riemannian manifolds. 

As a term the Cotton manifold or Cotton metric is not very common; it is used in \cite{govnur}. However, the Cotton tensor, introduced in \cite{cotton}, which gives the defining condition for such manifolds appears in wide variety of questions. We will explain below how Cotton manifolds are related to other classes of Riemannian manifolds.    

There are in fact at least 3 different definitions for quasi Einstein manifolds. The one we are interested in was apparently introduced in \cite{adamiy}. The quasi Einstein property  was then a special case of larger class of Riemannian manifolds. Other definitions, neither directly  related to the present article nor to each other can be found for example in \cite{chaval} and \cite{cashwe}. As the name suggest these spaces are typically related to problems in general relativity and the hence the metric in that case is typically not positive definite. 

We will analyze some connections of the above classes of manifolds. However, the main part of our paper is devoted to the construction of large  families examples of these different spaces. In the papers where these types of Riemannian manifolds are considered there are very few actual examples. In some sense this is natural since producing an example implies solving a relatively big system of PDE. Below we will show how to use the theory of overdetermined PDE (also called formal theory of PDE) \cite{dudsam,pommaret,werner} to produce solutions. The conclusion is in fact that with appropriate tools it is not particularly hard to find examples.

Below we have chosen examples more or less randomly with no particular application in mind. However, the reader who wants solutions of some specific form can easily adapt our approach to other contexts. Of course this approach does not always lead to explicit solutions, but the analysis can still give important information about the nature of solutions. In fact the special form of the system that is obtained in the analysis is even suitable for numerical computations, if one wants to explore numerically different possibilities.

The paper is organized as follows. 
In section 2 we recall some notions which are needed in the analysis. In section 3 the classes of Riemannian manifolds are introduced, and the relationships between them are analyzed. Then in section 4 we formulate our computational problems precisely. Finally in section 5 we present and discuss the examples and in section 6 there are some concluding remarks.

\section{Preliminaries}

\subsection{Geometry}
Let $M$ be a smooth $n$ dimensional manifold with Riemannian metric $g$. The pointwise norm of a tensor $T$ is denoted by $|T|$. The covariant derivative is denoted by $\nabla$. We say that a tensor $T$ is \emph{parallel}, if $\nabla T=0$. The curvature tensor is denoted by $R$ and the \emph{Ricci tensor} is $\mathsf{Ri}_{jk}=R^i_{ijk}$ and the scalar curvature is $\mathsf{sc}=\mathsf{Ri}^k_k$. There are several conventions regarding the signs and indices of curvature tensors. We will follow \cite{petersen}. 

In several places we will need the \emph{Ricci identity} which for general  tensors $A$ of type $(m, n)$  has the form
\begin{equation}
   A^{j_{1}\cdots j_{m}}_{i_{1}\cdots i_{n};ij}-A^{j_{1}\cdots j_{m}}_{i_{1}\cdots i_{n};ji}=\sum^{n}_{q=1}A^{j_{1}\cdots j_{m}}_{i_{1}\cdots i_{q-1}\ell i_{q+1}\cdots i_{n}}R^{\ell}_{iji_{q}}-\sum^{m}_{p=1}A^{j_{1}\cdots j_{p-1}\ell j_{p+1}\cdots j_{m}}_{i_{1}\cdots i_{n}}R^{j_{p}}_{ij\ell }
\label{general-ricci-identity}
\end{equation}
The Bianchi identity is
\begin{equation}
  R_{hijk;\ell}+R_{hik\ell;j}+R_{hi\ell j;k}=0
\label{bianchi}    
\end{equation}
By multiplying above equation on $g^{hk}$, we have
\[
    \mathsf{Ri}_{ij;\ell}-\mathsf{Ri}_{i\ell;j}=R^{h}_{\ell ji;h}
\]
which then implies that
\begin{equation}
    \mathsf{sc}_{;k}=2\,\mathsf{div}(\mathsf{Ri})=2\,\mathsf{Ri}^j_{k;j}
    \label{bia}
\end{equation}
Let us then define some classical tensors which are needed in the sequel. 
\begin{definition} Let $M$ be a $n$ dimensional Riemannian manifold with metric $g$. Then
\begin{itemize}
    \item[(i)] Schouten tensor is
    \[
      S= \mathsf{Ri}-\frac{\mathsf{sc}}{2n-2}\,g
    \]
    \item[(ii)] Cotton tensor is
    \[
       C_{ijk}=S_{ij;k}-S_{ik;j}
    \]
    \item[(iii)] Weyl tensor is
    \[
      W_{hijk} = R_{hijk }+\frac{\mathsf{sc}}{(n-1)(n-2)}\Big(g_{hk}g_{ij}-g_{hj}g_{ik}\Big)-\frac{1}{n-2}\Big(\mathsf{Ri}_{hk}g_{ij}-\mathsf{Ri}_{hj}g_{ik} + \mathsf{Ri}_{ij}g_{hk}-\mathsf{Ri}_{ik}g_{hj}\Big)\,.
    \]
\end{itemize}
\end{definition}
In some references Schouten tensor is some constant multiple of $S$ given above. Note that $W=0$ when $n\le 3$ and $S=0$ when $n=2$. Let us also recall
\begin{theorem}
Let $M$ be a $n$ dimensional Riemannian manifold. Then
\begin{itemize}
    \item[(1)] $M$ is conformally flat if and only if $C=0$ when $n=3$ or $W=0$ when $n\ge 4$.
    \item[(2)] $\mathsf{div}(W)=\frac{n-3}{n-2}\,C$ when $n\ge 4$.
\end{itemize}
\label{coflat}
\end{theorem}

\subsection{Determinantal varieties}

Let $\mathbb{R}^{m\times n}$ be the vector space of real  $m\times n$ matrices and let $V_r$ be the subvariety of matrices of rank at most $r$. This is a determinantal variety, defined by setting to zero all minors of size $(r+1)\times (r+1)$. There are thus $\binom{m}{r+1}\binom{n}{r+1}$  polynomials which generate the ideal defining $V_r$. However, not all generators are algebraically independent and one can show that
\[
  \mathsf{codim}(V_r)=(m-r)(n-r)
\]
Let then $\mathcal{S}_n$ be the vector space of real  $n\times n$ symmetric matrices and let  $V_{r}^s$ be the subvariety of symmetric matrices of rank at most $r$. The relevant ideal is now generated by $\binom{n}{r+1}^2$  polynomials, but due to symmetry we have much less independent generators, and one can show that in this case
\begin{equation}
    \mathsf{codim}(V_r^s)=\binom{n-r+1}{2}
    \label{codim}
\end{equation}

\subsection{Overdetermined PDE}
The existence of various manifolds  considered below depends on the solvability of certain systems of overdetermined PDE. For a general overview of overdetermined PDE we refer to  \cite{dudsam,pommaret,werner} and references therein. Both books contain also historical comments on the development of the subject which started at the end of 19th century. 

Very early it was realized that before one could actually prove any existence results for overdetermined PDE in some definite function space one should first analyze the structural properties of the system. The main difficulty of the analysis of overdetermined systems is related to integrability conditions: in other words by differentiating the equations one may find new equations which are algebraically independent of the original equations. The process of finding the integrability conditions is called \emph{completion}, and the goal was to find all integrability conditions. 

Analysis of the completion process lead to two complementary approaches: geometric and algebraic. The geometric approach is based on interpreting PDE as submanifolds of jet spaces. The algebraic approach requires that the nonlinearities are polynomial and hence the equations can be interpreted as differential polynomials and the systems themselves are differential ideals generated by the given differential polynomials. 

It turns out that proving that the system is complete, or that the system can actually be completed, is quite tricky and we simply refer again to \cite{dudsam,pommaret,werner} for details. In spite of this heavy machinery which is required for the theory the end result is perhaps surprisingly constructive: there are actual algorithms for computing the completed system, i.e. the system which contains all integrability conditions. The completed system is called the \emph{involutive system}, and the completion algorithm is usually known as Cartan-Kuranishi algorithm.  

The analysis of structural properties of overdetermined PDE is also called formal theory of PDE. The word formal appears because one can say that the involutive form of the system has solutions as formal power series. One can say that in the involutive system all relevant information about the system is explicit while in the initial system it was only implicit.

An analogous situation arises in polynomial algebra. A polynomial system generates an ideal which in turn defines the corresponding variety. Now computing  the Gr\"obner basis of the ideal gives a lot of information about the variety \cite{colios}. Intuitively one may think about computing the involutive form of a system of PDE like computing the Gr\"obner basis of an ideal. 

The idea of Gr\"obner bases can be generalized to differential equations, where equations are interpreted as differential polynomials \cite{liz}.  However, not all properties of Gr\"obner bases of the algebraic case carry over to the differential case. Anyway the ideas related to Gr\"obner bases and ideal theory in general  are present in the actual implementations of completion algorithms.

One final comparison to polynomial case is perhaps helpful. In the polynomial case any variety can be decomposed to a finite number of irreducible varieties which means that any polynomial ideal  is an intersection of finite number of primary ideals. This property is still valid in the differential context in the following form: any radical differential ideal is a finite intersection of prime differential ideals. Hence if the nonlinearities are polynomial, and they will be in all systems considered below, one may also try to find the decomposition of the  involutive form. Evidently finding this decomposition greatly facilitates any further analysis of the system. 

In what follows we will use the algorithm \textsf{rifsimp} which is described in detail in \cite{rifsimp}, see also \cite{greg}. The acronym \textsf{rif}  means \emph{reduced involutive form}.
This algorithm assumes that the nonlinearities are polynomial, and that the implied differential field is the field of rational functions. It 
can also handle inequations and it can compute the decomposition of the system.  

The algorithm is implemented as the command \textsf{rifsimp}  in \textsf{Maple}.\footnote{\textsf{https://www.maplesoft.com/}}  In setting up the systems of equations the \textsf{Differential Geometry} package of \textsf{Maple} was also very useful.

Finally we note that the word ''overdetermined'' is a bit misleading. This term is traditionally used, but actually one simply means the analysis of general PDE systems. The important concept is the involutivity (or some other canonical form), and in many (or even most)  cases it is not necessary to define precisely what is meant by the term overdetermined (or underdetermined). In particular below this definition is not needed. Also one can find (at least) two different definitions in the literature which are both reasonable in their ways; see \cite{dudsam} and \cite{werner} for these different definitions.

\section{Some properties and relationships between various classes}

In the following we will consider several classes of Riemannian manifolds. These classes are defined by requiring that the corresponding Ricci tensor satisfies some condition $\mathsf{P}$. In this case we can also say that the manifold or the Riemannian metric is of the type $\mathsf{P}$. Of course we will always assume that  $\mathsf{Ri}\ne0$. 

\begin{definition}
Ricci tensor is 
\begin{itemize}
    \item \textit{Ricci recurrent}, $\mathsf{RR}$, if there is a nonzero one form $\beta$ such that
\begin{equation}
  \mathsf{Ri}_{ij;\ell}=\beta_{\ell} \mathsf{Ri}_{ij}\ .
\label{prs-rirecon-1}
\end{equation}
\item \textit{pseudo Ricci symmetric}, $\mathsf{PRS}$, if there is a nonzero one form $\alfa$ such that
\begin{equation}
   \mathsf{Ri}_{ij;\ell}=2\alfa_{\ell}\mathsf{Ri}_{ij}+\alfa_{i}\mathsf{Ri}_{\ell j}+\alfa_{j}\mathsf{Ri}_{i\ell }
\label{prs-1}
\end{equation}
\item \textit{quasi Einstein}, $\mathsf{QE}$, if there are functions $a$ and $b$, and one form $\omega$ such that
\begin{equation}
   \mathsf{Ri}=a\,g+b\,\frac{\omega\otimes\omega}{|\omega|^2}  
\label{qe-def}
\end{equation}
If $b=0$ the Ricci tensor is Einstein.
\item \textit{Cotton}, $\mathsf{CO}$, if the Cotton tensor is zero.
\end{itemize}
\label{perus-def}
\end{definition}

In dimension two we have $\mathsf{Ri}=\kappa\, g$ where $\kappa$ is the Gaussian curvature. Hence any manifold is Einstein, Cotton and Ricci recurrent. On the other hand no manifold is pseudo Ricci symmetric. Hence from now on we suppose that the dimension $n\ge 3$.

It is clear that if $\mathsf{Ri}$ is parallel the \textsf{RR} condition cannot be satisified, and on the other hand the \textsf{CO} condition is always satisfied. Below we will see that \textsf{PRS} condition is incompatible with parallelism.  

If the metric satisfies the condition $\mathsf{div}(W)=0$ it is sometimes said to be \emph{conformally conservative}. Hence by Theorem \ref{coflat} Cotton manifolds are conformally conservative when $n\ge 4$ and conformally flat when $n=3$.
Finally recall that if $M$ is an Einstein manifold then $a=\mathsf{sc}/n=$constant. 

\subsection{Ricci recurrent}
Let us then start with the \textsf{RR} case. Multiplying \eqref{prs-rirecon-1} by $\mathsf{Ri}^{ij}$ we obtain
\[
 \beta= \frac{\mathsf{Ri}^{ij}\mathsf{Ri}_{ij;k}}{|\mathsf{Ri}|^2} =\tfrac{1}{2}\,\nabla\ln(|\mathsf{Ri}|^2)
\]
From this we get the following characterization.
\begin{lemma}
 Let $\mathsf{NRi}=\frac{\mathsf{Ri}}{|\mathsf{Ri}|}$. Then $\mathsf{Ri}$ is $\mathsf{RR}$ if and only if $\mathsf{NRi}$ is parallel.
\end{lemma}
\begin{proof}
Simply taking the covariant derivative of $\mathsf{NRi}$ we see that it is zero precisely when $\mathsf{Ri}$ is $\mathsf{RR}$ with $\beta$ as given above.
\end{proof}

However, it turns out that the \textsf{RR} case can be characterized purely in an algebraic way. In \cite{roter0} it is shown that actually
\begin{equation}
      \mathsf{Ri}^k_i\mathsf{Ri}^i_\ell=\tfrac{1}{2}\,\mathsf{sc}\,\mathsf{Ri}^k_\ell
     \label{ricci-alg}
\end{equation}
Note that this result crucially depends on the fact that $g$ is positive definite. 
 But this leads easily to the following characterization of the Ricci tensor.
\begin{theorem}
Suppose that $\mathsf{Ri}$ is recurrent. Then it has a double eigenvalue $\tfrac{\mathsf{sc}}{2}$ and eigenvalue zero of multiplicity $n-2$. Moreover
\[
\beta=\nabla \ln(\mathsf{sc})\quad, \quad
\mathsf{sc}^2=2\,|\mathsf{Ri}|^2\quad\mathrm{and}\quad
\mathsf{Ri}\,\beta=\tfrac{\mathsf{sc}}{2}\, \beta\ .
\]
\label{recu-char}
\end{theorem}
\begin{proof}
Let $\lambda_j$ be the eigenvalues of $\mathsf{Ri}$. Then the formula \eqref{ricci-alg} implies that
\[
    \lambda_k^2=\tfrac{1}{2}\,\big(\lambda_1+\dots+\lambda_n\big)\lambda_k\ .
\]
This gives the first statement. Since the scalar curvature cannot be zero the formula for $\beta$ is obtained multiplying \eqref{prs-rirecon-1} by $g^{ij}$. Taking the trace in  \eqref{ricci-alg}  gives the second formula. Then by formula \eqref{bia} we have
\[
  \mathsf{sc}\,\beta_\ell= sc_{;\ell}=2\,g^{hk}\mathsf{Ri}_{h\ell;k}=
   2\,g^{hk}\beta_k \mathsf{Ri}_{h\ell}=2\mathsf{Ri}\,\beta
\]
\end{proof}

From this we immediately get
\begin{corollary} Ricci tensor cannot be at the same time \textsf{RR} and \textsf{CO}.
\end{corollary}
\begin{proof} 
If the Ricci tensor satisfies the condition \textsf{RR} then 
\[
   S_{ij;k}=\frac{\mathsf{sc}_{;k}}{\mathsf{sc}}\Big( \mathsf{Ri}-\frac{1}{2(n-1)}\mathsf{sc}\,g_{ij}\Big)
   =\frac{\mathsf{sc}_{;k}}{\mathsf{sc}}\,S_{ij}
\]
Note that Schouten tensor is also ''recurrent'' in this case. 
Now by previous Theorem we know that $\mathsf{Ri}$ has an eigenvector $v$ such that $\mathsf{Ri}\,v=0$ and $g(\mathsf{grad}(\mathsf{sc}),v)=0$. This implies that
\[
   C_{ijk}v^j=S_{ij;k}v^j-S_{ik;j}v^j=
   -\tfrac{1}{2(n-1)}\mathsf{sc}_{;k}\,v_i\ne 0\ .
\]
\end{proof}

\subsection{Pseudo Ricci symmetric}
It turns out that the associated one form is almost the same in the $\mathsf{PRS}$ case.

\begin{lemma}
Let the Ricci tensor be $\mathsf{PRS}$ with associated one form $\alpha$. Then 
 $\mathsf{Ri}\,\alpha=0$ and 
 \[
  \alpha=\tfrac{1}{4}\,\nabla\ln(|\mathsf{Ri}|^2)
\]
If the scalar curvature is not zero then
\[
\alpha=\tfrac{1}{2}\,\nabla \ln(\mathsf{sc})\quad\mathrm{and}\quad
\mathsf{sc}^2=c\,|\mathsf{Ri}|^2
\]
for some constant $c$. 
\label{prs-closed}
\end{lemma} 
It is easy to see that if the scalar curvature is constant then it has to be zero. However, it seems to be an open problem if the case $\mathsf{sc}=0$ can actually occur. If one does not assume the postive definiteness of the metric then it is easy to construct examples with $\mathsf{sc}=0$. However, we were unable to find an example with a positive definite metric. 

\begin{proof}
Multiplying \eqref{prs-1} with $g^{i\ell}$ and using the formula \eqref{bia} we get
\[
 \mathsf{sc}_{;j}=2 g^{i\ell}\mathsf{Ri}_{ij;\ell}=4\alpha _{\ell}g^{i\ell}\mathsf{Ri}_{ij}+2g^{i\ell}\alpha_{i}\mathsf{Ri}_{\ell j}+2\alpha_{j}\mathsf{sc}=
 6\alpha^i\mathsf{Ri}_{ij}+2\alpha_{j}\mathsf{sc}
\]
On the other hand multiplying \eqref{prs-1} with $g^{ij}$  we get
\[
  \mathsf{sc}_{;\ell}=2\alpha_{\ell}\mathsf{sc}+2\alpha^i\mathsf{Ri}_{i\ell}
\]
Hence $ \mathsf{Ri}\,\alpha=0$.  The expression for $\alpha$ is then  obtained by multiplying  \eqref{prs-1} with $\mathsf{Ri}^{ij}$. If $\mathsf{sc}\ne0$ we have   $\mathsf{sc}_{;\ell}=2\alpha_{\ell}\mathsf{sc}$ which gives the other expression for $\alpha$.
\end{proof}

From this we get immediately
\begin{corollary}
\begin{itemize}
    \item[(i)] The Ricci tensor cannot be both $\mathsf{RR}$ and $\mathsf{PRS}$.
    \item[(ii)] If \textsf{Ri} is parallel then the \textsf{PRS} condition cannot be satisfied.
\end{itemize}
\label{rr-prs}
\end{corollary}
\begin{proof}
(i) The previous Lemmas imply that  the associated one forms in the two cases satisfy $\beta=2\alpha$. Hence
\[
   \beta_{\ell} \mathsf{Ri}_{ij}=
    \mathsf{Ri}_{ij;\ell}=2\alfa_{\ell}\mathsf{Ri}_{ij}+\alfa_{i}\mathsf{Ri}_{\ell j}+\alfa_{j}\mathsf{Ri}_{i\ell }
\]
implies that $\alfa_{i}\mathsf{Ri}_{\ell j}+\alfa_{j}\mathsf{Ri}_{i\ell }=0$. Multiplying this with $\alpha^i$ gives $|\alpha|^2\mathsf{Ri}=0$ which is impossible.

(ii) If \textsf{Ri} is parallel, then multiplying \eqref{prs-1}  with $\alpha^\ell$ gives $2|\alpha|^2 \mathsf{Ri}=0$ which is impossible.
\end{proof}

\subsection{Quasi Einstein}

Let us then analyze quasi Einstein structure. 
To this end it is convenient to formulate the condition \eqref{qe-def} differently. Let us introduce the tensor $T=\mathsf{Ri}_{ij}-a\,g_{ij}$. Then we can say
\begin{itemize}
    \item[(i)] $M$ is an Einstein manifold, if $T=0$  and $a=\mathsf{sc}/n$.
    \item[(ii)] $M$ is a quasi Einstein manifold, if  the matrix rank of $T$ is one.
\end{itemize}
In this way we see that the one form $\omega$ and the function $b$ are actually quite irrelevant in the analysis of the existence of quasi Einstein structure.

Now if we want that the symmetric tensor $T$ is of matrix rank one then according to the formula \eqref{codim} there are only $\binom{n}{2}$ algebraically independent differential equations. On the other hand there are  $1+\binom{n+1}{2}$ unknowns, namely $a$ and the  components of $g$, so it should not be too difficult to find solutions. 
Note that no derivatives of $a$ appear in the equations, so it is natural to first eliminate $a$ from the equations and then consider the system obtained for the metric $g$. This is the approach we will follow below. 

Of course when we suppose that $g$ is of specific form it is not a priori clear how many independent differential equations are obtained in this way.
Note that as a PDE system $\mathsf{QE}$ system is essentially different from $\mathsf{RR}$ and $\mathsf{PRS}$ systems. $\mathsf{QE}$ is a fully nonlinear system, i.e. nonlinear in highest derivatives while $\mathsf{RR}$ and $\mathsf{PRS}$ systems are quasilinear. 

Once the appropriate $T$ is found $b$ and $\omega$ can easily be computed. By taking traces  in the formula \eqref{qe-def} one sees immediately that $b=\mathsf{sc}-na$ and then $\omega$ can be solved from the linear system
\[
    T_{ij}\omega^j=\big(\mathsf{sc}-n\,a\big)\,\omega_i
\]
Still another way to characterize the $\mathsf{QE}$ case which is actually useful when considering examples is that $\mathsf{Ri}$ has a simple eigenvalue $\mathsf{sc}-(n-1)a$ corresponding to the eigenvector $\omega$, and all vectors orthogonal to $\omega$ are eigenvectors with eigenvalue $a$ whose multiplicity is thus $n-1$.
But this formulation immediately gives
\begin{theorem}
If the Ricci tensor is recurrent and $n=3$ then it is automatically quasi Einstein. If $n>3$ the Ricci tensor cannot be both recurrent and quasi Einstein.
\label{rr-qe}
\end{theorem}
\begin{proof}
By Theorem \ref{recu-char} the eigenvalue structure of  $\mathsf{Ri}$ can satisfy both $\mathsf{RR}$ and $\mathsf{QE}$ conditions only if $n=3$. In this case if $\omega$ is the eigenvector corresponding to the zero eigenvalue we can write the Ricci tensor as follows
\[
 \mathsf{Ri}_{ij}= \frac{\mathsf{sc}}{2} \Big(g_{ij}-\frac{\omega_{i}\omega_{j}}{|\omega|^{2}}\Big)\ .
\]
\end{proof}

On the other hand \textsf{PRS} and \textsf{QE} conditions can be satisfied in any dimension. In these cases the form of the Ricci tensor is as follows.
\begin{theorem} 
Let us suppose that both $\mathsf{PRS}$ and $\mathsf{QE}$ conditions are satisfied and let $\alpha$ be the associated one form. If $a\ne 0$ we have
\[
\mathsf{Ri}_{ij}= \frac{\mathsf{sc}}{n-1} \Big(g_{ij}-\frac{\alpha_{i}\alpha_{j}}{|\alpha|^{2}}\Big)
\]
If $a= 0$ then
\[
  g(\omega,\alpha)=0  \quad , \quad 
  \mathsf{Ri}_{ij}= \mathsf{sc}\,\frac{\omega_{i}\omega_{j}}{|\omega|^{2}}
   \quad \mathrm{and}\quad 
   \nabla_\alpha\omega=|\alpha|^2\omega\ .
\]
\label{q-prs}
\end{theorem}  
\begin{proof}
By multiplying formula (\ref{qe-def}) with $\alpha^{i}$ and applying Lemma \ref{prs-closed}, we have
\begin{equation}
  a\alpha_{j}+b\,\tfrac{g(\omega,\alpha)}{|\omega|^{2}}\omega_{j}=0
\label{q-prs-1}
\end{equation}
If $a\ne0$ and $g(\omega,\alpha)\neq 0$ then $\alpha$ and $\omega$ are linearly dependent and we may choose $\omega=\alpha$ which gives
\[
   a=\frac{\mathsf{sc}}{n-1} \quad \mathrm{and} \quad b=-a \ .
\]
When  $a=0$ the first two statements are obvious. To get the third we take the covariant derivative of the formula $ \mathsf{Ri}_{ij}= \mathsf{sc}\,\tfrac{\omega_{i}\omega_{j}}{|\omega|^{2}}$, and then multiply it with $\alpha^{i}$. Then using the formula (\ref{prs-1}) and simplifying we get the result. 
\end{proof}

\begin{lemma}
 Suppose that $\mathsf{Ri}$ satisfies the $\mathsf{PRS}$ condition and that  $\mathsf{sc}\ne0$; then
 \begin{itemize}
     \item[(i)] if $\mathsf{Ri}$ is also $\mathsf{QE}$ with $a\ne0$ then it is $\mathsf{CO}$
     \item[(ii)] if $\mathsf{Ri}$ is also $\mathsf{CO}$  then it is $\mathsf{QE}$
 \end{itemize}
 \label{prs-cc-qe}
\end{lemma}
\begin{proof}
If the \textsf{PRS} condition is satisfied and $\mathsf{sc}\ne 0$ then the Cotton tensor is
\[
  C_{ijk}=\frac{\mathsf{sc}_{;k}}{2\,\mathsf{sc}}\Big(\mathsf{Ri}_{ij}-\frac{\mathsf{sc}}{n-1}\,g_{ij}\Big)-
  \frac{\mathsf{sc}_{;j}}{2\,\mathsf{sc}}\Big(\mathsf{Ri}_{ik}-\frac{\mathsf{sc}}{n-1}\,g_{ik}\Big)
\]
Now simply substituting the expression for $\mathsf{Ri}$ given in Theorem \ref{q-prs} shows that $C=0$ which proves the statement (i). 

On the other hand it is easy to check that the Cotton tensor of the above form can be zero only if the matrix rank of $\mathsf{Ri}_{ij}-\frac{\mathsf{sc}}{n-1}\,g_{ij}$ is one which is precisely the \textsf{QE} condition.
\end{proof}

There is a completely different way to construct \textsf{QE} metrics which we now describe. The form of the condition makes one think about conformal equivalence. Let $g$ be a given metric and let $\hat g=\exp(2\lambda) g$ be a metric that is conformally equivalent to it. If now $\widehat{\mathsf{Ri}}$ is the Ricci tensor associated to $\hat g$ then
\[
  \widehat{\mathsf{Ri}}_{ij}=\mathsf{Ri}_{ij}-(n-2)\big(\lambda_{;ij}-\lambda_{;i}\lambda_{;j}\big)
    -\big(\Delta\lambda+(n-2)|\nabla\lambda|^2\big)g_{ij}
\]
Hence if we can find a metric $g$ and a function $\lambda$ such that $\mathsf{Ri}_{ij}=(n-2)\lambda_{;ij}$ then $\widehat{\mathsf{Ri}}$ is quasi Einstein:
\[
 \widehat{\mathsf{Ri}}_{ij}=
    -\big(\Delta\lambda+(n-2)|\nabla\lambda|^2\big)
    \exp(-2\lambda)\hat g_{ij}+(n-2)\lambda_{;i}\lambda_{;j}
\]
Note that the equations are trivially satisfied if $\mathsf{Ri}=\nabla\nabla\lambda=0$. However, if $\nabla\lambda\ne0$ there is still a nontrivial solution
\[
 \widehat{\mathsf{Ri}}_{ij}=
    -(n-2)|\nabla\lambda|^2
    \exp(-2\lambda)\hat g_{ij}+(n-2)\lambda_{;i}\lambda_{;j}
\]
The existence of solutions to PDE system  $\mathsf{Ri}=T$ where $T$ is a given symmetric tensor is analyzed in \cite{deturck}. It is instructive to consider this system from the point of view of overdetermined systems. 

At the outset there are thus $\tfrac{1}{2}\,n(n+1)$ second order quasilinear PDE for the components of the metric so that it seems the system is determined. However, we have the Bianchi identity \eqref{bia} which should be taken into account. Let us then define the Bianchi operator $B$ by the formula
\[
   B(T)=2\,\mathsf{div}(T)-\nabla \mathsf{tr}(T)=
      2g^{ij} T_{ik;j}-g^{ij}T_{ij;k}
\]
Hence if the system $\mathsf{Ri}=T$ has solutions then necessarily the condition $B(T)=0$ must also be satisfied. Note that this is a system of $n$ first order equations in metric $g$. Let us then define the modified Bianchi operator
\[
  \tilde  B(T)=
      2g^{ij} T_{ik,j}-g^{ij}T_{ij,k}
\]
Note that the standard and covariant derivatives agree up to lower order corrections. Hence the components of $\tilde B(\mathsf{Ri})$ are second order differential operators. This leads to the following system:
\[
   \begin{cases}
   \mathsf{Ri}=T\\
   \tilde B(\mathsf{Ri})= \tilde  B(T)\\
   B(T)=0
   \end{cases}
\]
Interestingly the initial system $\mathsf{Ri}=T$ is not elliptic while the above completed system is, provided that $\mathsf{Ri}$ is of full (matrix) rank. The ellipticity can then be used to show the existence of local solutions. In our case the system is 
\begin{equation}
   \begin{cases}
   \mathsf{Ri}=(n-2)\lambda_{;ij}\\
   \tilde B(\mathsf{Ri})= (n-2)\tilde  B(\lambda_{;ij})\\
   B(\lambda_{;ij})=0
   \end{cases}    
   \label{qe2-syst}
\end{equation}
However, from the point of view of existence of solutions this is a very different system since $\lambda$ is also unknown, and moreover this is a third order system.

\section{Setting up the computational problem}

We will choose some families of metrics and try to construct metrics which satisfy one or more of the conditions in Definition \ref{perus-def}. All conditions lead to systems of PDE whose nonlinearities are polynomial, and hence we can use \textsf{rifsimp} to analyse them. 

It is known that the complexity of computing Gr\"obner basis is very bad (doubly exponential) in the worst case. Of course computing the involutive form is even more difficult. However, typically the time required for these computations is far from the worst case. On the other hand since there is no reasonable probability measure in the ''space of all problems'' there are no rigorous results on ''average'' complexity. Anyway all the solutions given below were obtained usually in few seconds and in any case in less than a minute with standard PC.
Note that the decomposition of the system may take a lot more time than computing just the ''most general'' solution.

In all cases there actually were several components in the system, so that in the language of differential algebra the initial system was never prime differential ideal.  It is not quite clear how to interpret this geometrically. Of course as components they provide essentially different solutions to PDE systems. From this it does not necessarily follow that the corresponding Riemannian manifolds are essentially different (i.e. not isometric). We did not attempt to study this problem. Since there were so many different components in all we will mostly give below only the most general one. 

Let us now formulate more precisely the PDE systems that we are trying to solve.

\begin{problem} (\textsf{RR} problem) 
Find a metric $g$ such that $P_{ijk}=0$ where 
\begin{equation}
    P_{ijk}= \mathsf{sc}\,\mathsf{Ri}_{ij;k}-\mathsf{sc}_{;k}\mathsf{Ri}_{ij}
    \label{rr-cond}
\end{equation}
 This is a third order quasilinear   system of $\tfrac{1}{2}\,n^2(n+1)$ PDE.
 \label{prob-rr}
\end{problem}

\begin{problem} (\textsf{PRS} problem) 
Find a metric $g$ such that $Q_{ijk}=0$ where 
\begin{equation}
    Q_{ijk}= 2\mathsf{sc}\,\mathsf{Ri}_{ij;k}-2\mathsf{sc}_{;k}\mathsf{Ri}_{ij}-\mathsf{sc}_{;i}\mathsf{Ri}_{k j}-\mathsf{sc}_{;j}\mathsf{Ri}_{ik }
    \label{prs-cond}
\end{equation}
 This is a third order quasilinear  system of $\tfrac{1}{2}\,n^2(n+1)$ PDE.
 \label{prob-prs}
\end{problem}

\begin{problem} (\textsf{CO} problem) 
Find a metric $g$ such that the Cotton tensor $C=0$. 
 This is a third order quasilinear   system of $n^2(n-1)$ PDE.
 \label{prob-cc}
\end{problem}

\begin{problem} (First \textsf{QE} problem) 
Find a metric $g$ and a function $a$ such that the matrix rank of $T=\mathsf{Ri}-a\,g$ is one. According to \eqref{codim} there are $\binom{n}{2}$ algebraically independent fully nonlinear second order equations. 
 \label{prob-qe1}
\end{problem}

\begin{problem} (Second \textsf{QE} problem) 
Find a metric $g$ and a function $\lambda$ such  that $\mathsf{Ri}=(n-2)\nabla\nabla \lambda$. Here we have $\tfrac{1}{2}\,n(n+1)$ quasilinear second order equations in $\tfrac{1}{2}\,n(n+1)+1$ unknowns. Note that in the completed system \eqref{qe2-syst} there are integrability conditions which are expressed in terms of (modified) Bianchi operator. However, there is no need to compute them explicitly since \textsf{rifsimp} takes care of them automatically.
 \label{prob-qe2}
\end{problem}

Note that the numbers of equations given above is the maximum number of algebraically independent PDE in the initial system. This number is achieved, if the metric is assumed to be completely general. However, if we assume that the metric is of specific form we may initially have less equations. 

But the number of equations is in fact not really important in the present context. We recall that also with polynomial ideals the number of generators of the ideal does not matter, and even the number of generators of the Gr\"obner basis does not give any useful information. In the same way it is very convenient when using the algorithm \textsf{rifsimp} that it is not necessary to check beforehand if the equations are actually algebraically independent; \textsf{rifsimp} takes care of that automatically. Hence in practice we simply compute the relevant tensor and then require that all of its components are zero; if there are some redundant equations in the system  then  \textsf{rifsimp} simply discards them. 

In the following we will look for the solutions of these systems. Since our analysis is local we will always consider only the situation in a single coordinate system. All our examples are of the form ''separation of variables''; in other words all unknown functions are of functions of one variable only. In this way our PDE systems reduce to ODE systems, and it is thus easier to find solutions. The actual form of the initial guess of the metrics is not very critical. Experimenting with different choices showed that it is not particularly hard to find nontrivial solutions. Hence the reader can easily modify our examples and find other solutions.

In the actual computations the use of inequations was quite convenient. We may just  look for those solutions where some function or a more complicated expression is nonzero. This is very natural in the problems below if for example we are not interested in cases where some unkonwn functions vanishes, or that the differential of the scalar curvature vanishes, or that $\mathsf{Ri}=0$. This can be speed up significantly the computations because then \textsf{rifsimp} does not need to worry about irrelevant subcases. 

Let us now briefly describe the output of \textsf{rifsimp}. The algorithm tries to express highest ranking derivatives in terms of lower ranking derivatives. Let $f=(f_1,\dots,f_k)$ be the unknown functions with $x=(x^1,\dots,x^n)$ as independent variables. Let $\alpha^j$ be some multiindices. Then the first part of output is as follows:
\begin{equation}
    \partial^{\alpha^j}f_j=F_j(x,f,\dots)\quad,\quad
    1\le j\le m  
    \label{output}
\end{equation}
In the arguments of $F_j$ there are only derivatives of lower ranking than $\partial^{\alpha^j}f_j$. The number $m$ is not known a priori. \textsf{rifsimp} also tries to eliminate the components $f_j$ from equations as far as possible. In linear algebra the Gaussian elimination reduces the problem to upper triangular form. Of course ''differential nonlinear upper triangular form'' does not exist in general, but \textsf{rifsimp} tries to compute a representation which is as close to it as possible. In the examples below it is seen clearly how this works. Note that the representation may depend heavily on the ranking chosen: the number $m$,  multiindices $\alpha^j$ and functions $F_j$ are not intrinsic.

When the problems are nonlinear all the relevant information about the system cannot always be expressed as in \eqref{output}. In these cases there are additional equations, called \emph{constraints} in \textsf{rifsimp}, of the form
\[
  H_\ell(x,f,\dots)=0\quad,\quad
    1\le \ell\le s  
\]
where the highest ranking derivatives of the arguments of $H_\ell$ are present nonlinearly. Below we will see an example  of this case also.

In addition to this there may be certain inequations in the output. When computing $F_j$ and $H_\ell$ sometimes one has to make some decisions if certain expressions are zero or not. This is how the system decomposes: in the generic case one assumes that ''typically'' any expression is nonzero. The algorithm keeps track of these assumptions and gives them in the output. But of course assuming that some expression is zero can produce solutions which are not contained in the ''general'' solution. Potentially there can be a lot of these branch points so that computing the decomposition, called \emph{casesplit} in \textsf{rifsimp}, can take much more time than computing the generic solution.

Note that the output of \textsf{rifsimp} has a lot of structure and contains a lot of information. So even if one is unable to actually explicitly solve the equations given in the output one typically can immediately obtain some important facts which characterize the possible solution set. In many cases considered below the output can even be easily used for numerical computations while it is not at all clear how a numerical solution could be computed using the initial system. 

Now the fact that \textsf{rifsimp} tries to approach the ''upper triangular form''  makes also the explicit solution easier. One can first solve equations with fewer variables, and then substitute these solutions to equations which contain more variables, like back substitution in Gaussian elimination. In solving the equations we often used the command \textsf{dsolve} in {\sc Maple}.

\section{Results}

\subsection{3 dimensional case}

Let us consider the following metric:
\begin{equation}
  g=f_1(x^1)(dx^{1})^{2}+f_2(x^1)h_2(x^2) (dx^{2})^{2}
  +f_3(x^1)h_3(x^2)q(x^3) (dx^{3})^{2}    
  \label{metric-3d}
\end{equation}

\begin{example}
Problem \ref{prob-rr} with \eqref{metric-3d}. 

In this  case our PDE system has a priori 18 independent equations but actually we have only 14 (not necessarily independent) nonzero equations. The system splits into seven subsystems. However, six systems either require that some unknown functions are constants, or the corresponding solutions give only the trivial solution where $\beta$ reduces to zero. Note that it is anyway possible that in those cases there are solutions which are not special cases of the one given below.

The remaining component of the system has three differential equations; first two equations for $f_j$:
\begin{align*}
    f_2''=&\frac{f_2'(f_1'f_2+f_2'f_1)}{2f_1f_2}\\
     f_3''=&\frac{f_2f_3f_1'f_3'+2f_1f_2(f_3')^2-f_1f_3f_2'f_3'}{2f_1f_2f_3}
\end{align*}
Evidently now one can give $f_1$ arbitrarily and then solve the remaining functions. However, one can actually eliminate one of the functions by solving $f_1$ and $f_3$ in terms of $f_2$ which gives the following family of solutions:
\[
    f_1=\frac{c_2(f_2')^2}{f_2}\quad\mathrm{and}\quad
    f_3=c_1 f_2^m\ .
\]
Note that $m$ need not be an integer. Then we have the third differential equation which contain $h_j$ and $f_j$. However, when we substitute the above formulas the functions $f_j$ disappear and we are left with
\[
  h_3''=\frac{(2m-1)c_2h_2(h_3')^2+mc_2h_3h_2'h_3'-m^2h_2^2h_3^2}{2mc_2h_2h_3}
\]
Solving this for $h_2$ yields
\[
   h_2=\frac{c_2h_3^{1/m}(h_3')^2}{h_3^2(c_2c_3-m^2h_3^{1/m})}=
   \frac{c_2m^2(h')^2}{h(c_2c_3-m^2h)}
\]
where we have introduced a new function $h_3=h^m$.Then writing $f$ instead of $f_2$ we can write our final metric as
\[
  g=\frac{c_2(f')^2}{f} (dx^{1})^{2}+\frac{m^2c_2f (h')^2}{h(c_2c_3-m^2h)} (dx^{2})^{2}+c_1f^mh^mq (dx^{3})^{2}
\]
Clearly one can choose constants and functions such that $g$ is positive definite. For scalar curvature we get $\mathsf{sc}=(1-m)c_3/(2mfh)$ and thus $\beta=-\nabla\ln(fh)$. Note that $m\ne 1$ because otherwise also $\mathsf{Ri}=0$.

\label{ex-rirec-1}
\end{example}

\begin{example}
Problem \ref{prob-prs} with \eqref{metric-3d}. 

 Now we have our PDE system which is very similar to $\mathsf{RR}$ case, but of course the solutions are different, by Theorem \ref{rr-prs}. 
Again we have 14 PDE, and computing with \textsf{rifsimp} get three cases where $\alpha\ne 0$. One case is the following:
\begin{align*}
        f_2'=&\frac{f_2f_3'}{f_3}\\
     f_3''=&\frac{f_1(f_3')^2+f_3f_1'f_3'}{2f_1f_3}\\
     h_3'''=&\frac{h_2h_3^2(h_2''h_3'+3h_2'h_3'')+4h_2^2h_3h_3'h_3''-h_2^2(h_3')^2-2h_2h_3h_2'(h_3')^2-2h_3^2(h_2')^2h_3'}{2h_2^2h_3^2}
\end{align*}
It turns out that in the first two equations we can  solve $f_1$ and $f_2$ in terms of $f_3$, and in the last one we get $h_2$ in terms of $h_3$:
\begin{align*}
    f_1=&\frac{c_1(f_3')^2}{f_3}\\
    f_2=&c_2f_3\\
    h_2=&\frac{(h_3')^2}{h_3(c_3h_3+c_4)}
\end{align*}
Then writing $f_3=f$ and $h_3=h$ we get
\begin{align*}
  g=&\frac{c_1(f')^2}{f} (dx^{1})^{2}+ \frac{c_2f(h')^2}{h(c_3h+c_4)} (dx^{2})^{2}+fhq (dx^{3})^{2}
  \quad\mathrm{and}\\
  \alpha=&-\frac{f'}{2f}\,dx^1\quad,\quad 
  \mathsf{Ri}=-\frac{c_1c_3+c_2}{4c_1c_2h}\Big(\frac{c_2(h')^2}{c_3h+c_4} (dx^{2})^{2}+
  qh^2 (dx^{3})^{2}\Big)\ .    
\end{align*}
Note that the solution obtained satisfies also the \textsf{QE} condition with $a\ne0$, and hence also the \textsf{CO} condition by Lemma \ref{prs-cc-qe}.

\label{ex-prs-3d}
\end{example}

\begin{example}
Problem \ref{prob-cc} with metric \eqref{metric-3d}.

Now we have 8 equations in the system. By computing with \textsf{rifsimp} get three cases and in the most general case we have
\begin{align*}
    f_2'''=&F(f_1,f_2,f_3)\\
   h_3''=&H(f_1,f_2,f_3,h_2,h_3)
\end{align*}
where $F$ and $H$ are very complicated expressions, involving also the derivatives of its arguments, so that we do not write them down explicitly. The function $H$ at the outset depends on $f_j$ but if $f_j$ satisfy the first  equation then actually $H$ does not depend on $x^1$. The dependence of $H$ on $f_j$ is only through initial conditions of the first  equation, and consequently by standard theorems we have the local solution, and the above system can even be used for numerical computations. 

However, it turns out that one can describe the solution in a more explicit way. One can actually solve the first equation for $f_3$ using quadratures which yields 
\[
   f_3(x^1)=\exp\big(\hat F(f_1,f_2,f_1',f_2',f_2'')\big)
\]
where in the expression $\hat F$ there are also some integrals whose integrands depend on the variables indicated. Now substituting this expression to the second equation gives
\[
   2h_2h_3h_3''-h_3h_2'h_3'-2h_2(h_3')^2
   +c_0h_2^2h_3^2=0
\]
where $c_0$ is constant. Solving this gives
\[
   h_2=\frac{(h_3')^2}{(c_1-c_0\,\ln(h_3))h_3^2}
\]
Hence we can choose $f_1$, $f_2$, $h_3$ and $q$ freely and it is clear that this choice, and the choice of constants $c_0$ and $c_1$, can be done in such a way that the metric is positive definite.
\label{ex-cc-3d}
 \end{example}

\begin{example}
Problem \ref{prob-qe1} with metric \eqref{metric-3d}.

Constructing the appropriate PDE system we obtain five nonzero equations. Since the function $a$ appears algebraically and in some equations even linearly we can solve it and substitute back to the equations. Note that one gets different families of solutions, depending on the choice of $a$. However, we will analyze only one particular family of solutions. 

After choosing $a$ we are thus left with one single PDE;  \textsf{rifsimp} gives then us the following system:
\begin{align*}
    f_2'''=&F_1(f_1,f_2,f_3)\\
    f_3'''=&F_2(f_1,f_2,f_3)\\
    h_3''=&H(f_1,f_2,f_3,h_2,h_3)
\end{align*}
The expressions for $F_j$ and $H$ are again so big that we do not give them explicitly. Also the dependence of $H$ on $f_j$ is only through initial conditions as in the previous example, so that choosing $f_1$ and $h_2$ arbitrarily yields an ODE system in the standard form.

However, we can also solve the system explicitly; denoting $f_2=f$ the first two equations give
\begin{align*}
    f_1=&\frac{c_1c_3(f')^2}{f_3 f}\\
    f_3=&c_3m^{-m}f^{1-m}\big(c_2f-1\big)^m
\end{align*}
Substituting this into third equations yields
\[
  h_3''=\frac{\big((3m-2)h_2h_3'+2(m-1)h_3h_2'\big)h_3'}{4(m-1)h_2h_3}
\]
Denoting $h_3=h$ and solving for $h_2$ yields
\[
   h_2=c_4h^{(2-3m)/(2m-2)}(h')^2
\]
After this it is straightforward to compute $a$, $b$ and $\omega$ which gives
\begin{align*}
    a=&\frac{m}{8(1-m)c_4f}\, h^{(2-m)/(2m-2)}-\frac{c_2^2f^2+(m-2)c_2f+(m-1)^2}{2c_1m^m f^m}\,\big(c_2f-1\big)^{m-2}\\
    b=&\frac{m}{8(m-1)c_4f}\, h^{(2-m)/(2m-2)}+\frac{m(m-1)}{2c_1m^m f^m}\,\big(c_2f-1\big)^{m-2}\\
    \omega=&fh'(c_2f-1)\partial_{x^1}+(2-2m)hf'\partial_{x^2} 
\end{align*}
\label{ex-qe1-3d}
\end{example}

\begin{example}
Problem \ref{prob-qe2} with metric \eqref{metric-3d}.

Here we see that $\lambda_2$, $h_3$ and $f_2$ must be constants; for simplicity let us choose $\lambda_2=h_3=f_2=1$. Then for other functions we obtain
\begin{align*}
    f_1=&c_1(\lambda_1')^2\\
    f_3=&c_2\lambda_1^2\\
    q_3=&\frac{c_1(\lambda_3')^2}{c_1c_3-c_2\lambda_3^2}
\end{align*}
Note that there is no condition on $h_2$, and also $\lambda_1$ and $\lambda_3$ can be freely chosen. This solution implies that  $\mathsf{Ri}=\nabla\nabla\lambda=0$, but of course $\widehat{\mathsf{Ri}}$ gives a nontrivial example of QE manifold.
\label{ex-qe2-3d}
 \end{example}

\subsection{First 4 dimensional case}

 Let us then consider a simple four dimensional case:
 \begin{equation}
  g=(dx^{1})^{2}+f(x^{1})q(x^4)\big((dx^{2})^{2}+(dx^{3})^{2}+(dx^{4})^{2}\big)    
  \label{metric-4d1}
 \end{equation}

\begin{example}
Problem \ref{prob-rr} with metric \eqref{metric-4d1}.

It turns out that $\mathsf{RR}$ system has only solutions with $\mathsf{Ri}=0$ so there are no examples of this form. Here the use of inequations was very convenient. When one added to the PDE system the condition $\beta\ne0$, \textsf{rifsimp} concluded that the system is inconsistent.
\label{ex-rr-4d1}
 \end{example}

\begin{example}
Problem \ref{prob-prs} with metric \eqref{metric-4d1}.

This illustrates quite well how \textsf{rifsimp} handles the system and how the solutions can split into several (in this case two) families, so that we describe this in more detail.  
The equations of the $\mathsf{PRS}$  system give
\begin{equation}
\begin{aligned}
 &   f''=\frac{(f')^2}{2f}\\
&    q'''=\frac{4qq'q''-3(q')^3}{q^2}\\
  &  4fq^2(q'')^2-10fq(q')^2q''+6f(q')^4+2q^4(f')^2q''-3q^3(f')^2(q')^2=0
\end{aligned}
\label{yht-4d-prs}  
\end{equation}
Note that the last equation depends on both variables $x^1$ and $x^4$ so it seems that there might not be solutions. The third equation is different from the other two in another way. In first two equations highest order derivative is explicitly given in terms of lower order derivatives. In the final equation the highest derivative, namely $q''$ appears non linearly and consequently cannot be given explicitly. Recall that \textsf{rifsimp} calls equations of this type constraints. 

Anyway when solving the first equation we obtain $f=(c_1x^1+c_0)^2$ and substituting this to the final equation gives
\[
   2f\big( 2qq''-3(q')^2\big)\big(2c_1^2q^3+qq''-(q')^2\big)=0
\]
Now taking the first factor we have
\begin{equation}
\begin{aligned}
 &   2qq''-3(q')^2=0\quad\Rightarrow\quad
    q=\frac{1}{(c_2x^4+c_3)^2}\\
  &  \alpha=-\frac{c_1}{c_1x^1+c_0}\,dx^1\\
  & \mathsf{Ri}=-\frac{2(c_1^2+c_2^2)}{(c_2x^4+c_3)^2}\Big((dx^{2})^{2}+(dx^{3})^{2}+(dx^{4})^{2}\Big) 
\end{aligned}
\label{ri-prs-1}
\end{equation}

while the second factor yields
\begin{equation}
\begin{aligned}
 &  2c_1^2q^3+qq''-(q')^2=0\quad\Rightarrow\quad
    q=\frac{c_2^2}{c_1^2(\cosh(c_2x^4+c_3))^2}\\
  &  \alpha=-\frac{c_1}{c_1x^1+c_0}\,dx^1+c_2\tanh(c_2x^4+c_3)dx^4\\
  & \mathsf{Ri}=-c_2^2\big((dx^{2})^{2}+(dx^{3})^{2}\big)
\end{aligned}   
\label{ri-prs-2}
\end{equation}

Note that both families of solutions also satisfy the second equation of the system \eqref{yht-4d-prs}. 

 The solution \eqref{ri-prs-1} satisfies  also the $\mathsf{QE}$ condition and hence by Lemma \ref{prs-cc-qe} also the condition \textsf{CO}.  On the other hand the Ricci tensor corresponding to the solution \eqref{ri-prs-2} has two double eigenvalues and hence cannot be $\mathsf{QE}$.

\label{ex-prs-4d1}
 \end{example}

\begin{example}
Problem \ref{prob-cc} with metric \eqref{metric-4d1}.

The \textsf{CO} case is very easy: $f$ is arbitrary and 
\[
    q''=\frac{3(q')^2}{2q}\quad \Rightarrow\quad
     q=\frac{1}{(c_{1}x^{4}+c_0)^{2}}
\] 
Here we get the same $q$ as in \eqref{ri-prs-1}. Hence here \textsf{PRS} case is a subcase of \textsf{CO} case.
\label{ex-cc-4d1}
 \end{example}

\begin{example}
Problem \ref{prob-qe1} with metric \eqref{metric-4d1}.

After computing the minors we solve $a$ from one of the equations which yields
\begin{align*}
  a=-\frac{2f''q^{3}f+(f')^{2}q^{3}+2fqq''-(q')^{2}f}{4f^{2}q^{3}}
\end{align*}
Substituting this expression to the system and then computing with \textsf{rifsimp} gives the following system for $f$ and $q$:
\begin{align*}
    f'''=-&\frac{-2ff''f'+(f')^{3}}{f^{2}}\\
    q''=&\frac{4f''q^{3}f-4(f')^{2}q^{3}+(q')^{2}f}{2qf}
\end{align*}
The first equation can be solved and then the second one is in the standard form:
\begin{align*}
    f=&\,4c_3\cosh(c_1x^1+c_2)^2\\
    q''=&\frac{32c_{1}^{2}c_{3}q^{3}+(q')^{2}}{2q}
\end{align*}
Interestingly if $4c_1^2c_3=1$ then $q$ is a Weierstrass elliptic function. 

\label{ex-qe1-4d1}
 \end{example}

\begin{example}
Problem \ref{prob-qe2} with metric \eqref{metric-4d1}.

Here it is natural to suppose that $\lambda=\lambda_1(x^1)\lambda_4(x^4)$. The most general solution is now incompatible with the positive definiteness of the metric. However, we still have a nontrivial family solutions. First we set $\lambda_4=1$ and then compute $q=1/(c_1x^4+c_0)^2$. So here again $q$ must be the same as in \textsf{CO} case and in one of the \textsf{PRS} cases. Hence all metrics obtained in this way must satisfy also the \textsf{CO} condition.

Substituting the computed value of $q$ to the system  leaves us with the following equations:
\begin{align*}
  &  2ff''+ff'\lambda_1'+8c_1^2f+(f')^2=0\\
  &  4f^2\lambda_1''-6ff'\lambda_1'-12c_1^2f-3(f')^2=0
\end{align*}
There is no explicit formula for solution but again by standard theorems the local solution exists. Note that the metrics obtained in this way  satisfy the \textsf{QE} condition. Hence  we have metrics $g$ and $\hat g$ which are both \textsf{QE} and which are conformally equivalent. 
\label{ex-qe2-4d1}
 \end{example}

\subsection{Second 4 dimensional case}

Let us now consider
\begin{equation}
  g=f_1(x^1)(dx^{1})^{2}+f_2(x^{1})(dx^{2})^{2}+f_3(x^1)(dx^{3})^{2}+f_4(x^1)(dx^{4})^{2}
\label{metric-4d2}
\end{equation}

\begin{example}
Problem \ref{prob-rr} with metric \eqref{metric-4d2}.

It turns out that $\mathsf{RR}$ conditions force two of the functions $f_2$, $f_3$ and $f_4$ to be constants. Choosing for example $f_3=f_4=1$ we have $\beta=\nabla\ln(\mathsf{sc})$ where
\[
 \mathsf{sc}=\frac{-2f_1f_2f_2''+f_1(f_2')^2+f_2f_1'f_2'}{2f_1^2f_2^2}
\]
In essence the problem reduces to the 2 dimensional case and of course in 2 dimensions any metric is \textsf{RR}.
 \label{ex-rr-4d2}
 \end{example}

\begin{example}
Problem \ref{prob-prs} with metric \eqref{metric-4d2}.

 The system decomposes to 8 components and  the most general one gives the following system:
\begin{align*}
f_2''=&\frac{\big(f_2f_3f_4f_1'+ f_1f_3f_4f_2'-f_1f_2f_4f_3'-f_1f_2f_3f_4'\big)f_2'}{2f_1f_2f_3f_4}\\
f_3''=&\frac{\big(f_2f_3f_4f_1'- f_1f_3f_4f_2'+f_1f_2f_4f_3'-f_1f_2f_3f_4'\big)f_3'}{2f_1f_2f_3f_4}\\
f_4''=&\frac{f_2f_3f_4f_1'f_4'+2f_1(f_4)^2f_2'f_3'+ f_1f_3f_4f_2'f_4'+f_1f_2f_4f_3'f_4'+f_1f_2f_3(f_4')^2}{2f_1f_2f_3f_4}
\end{align*}
It turns out that one can solve this explicitly. Let us set $f_3=f$; then the other functions are given by
\begin{align*}
    f_1&=c_1c_2c_4m^2\exp(c_3 f^{m+1})f^{(m^2-m-1)/(m+1)}(f')^2\\
    f_2&=c_2f^m\\
    f_4&=c_4\exp(c_3 f^{m+1}) f^{-m/(m+1)}
\end{align*}
This yields
\begin{align*}
   \alpha&=\frac{\big(m-c_3(m+1)^2f^{m+1}\big)f'}{2(m+1)f}\,dx^1\\
  \mathsf{Ri}&=-\frac{c_3(m+1)^2}{2c_1c_2m^2}\, (dx^4)^2
\end{align*}
Note that $\mathsf{Ri}$ also satisfies  the \textsf{QE} condition  with $a=0$, see Theorem \ref{q-prs}. 
 
 \label{ex-prs-4d2}
 \end{example}

\begin{example}
Problem \ref{prob-cc} with metric \eqref{metric-4d2}.

In the \textsf{CO} case we have initially 6  equations, but  \textsf{rifsimp}  gives only the following 2 equations:
\begin{align*}
   f_2'''&=F_{2}(f_{1},f_{2},f_{3},f_{4})
 \\
   f_3'''&=F_{3}(f_{1},f_{2},f_{3},f_{4})
\end{align*}   
Here the expressions of $F_2$ and $F_3$ are so big that we do not give them explicitly. Anyway choosing $f_1$ and $f_4$ arbitrarily we have a standard ODE system for $f_2$ and $f_3$.

 \label{ex-cc-4d2}
 \end{example}

\begin{example}
Problem \ref{prob-qe1} with metric \eqref{metric-4d2}.

There are three families of solutions and as usual we give the most general. First \textsf{rifsimp} gives:
\begin{align*}
f_2''=&\frac{2f_1f_2^2f_3f_4f_4''-f_1f_2^2f_3(f_4')^2+f_1f_2^2f_4f_3'f_4'-f_1f_2f_4^2f_2'f_3'+f_1f_3f_4^2(f_2')^2-f_2^2f_3f_4f_1'f_4'+f_2f_3f_4^2f_1'f_2'}{2f_1f_2f_3f_4^2}\\
f_3''=&\frac{2f_1f_2f_3^2f_4f_4''-f_1f_2f_3^2(f_4')^2+
f_1f_2f_4^2(f_3')^2-f_1f_3f_4^2f_2'f_3'+
f_1f_3^2f_4f_2'f_4'-f_2f_3^2f_4f_1'f_4'+
f_2f_3f_4^2f_1'f_3'}{2f_1f_2f_3f_4^2}
\end{align*}
If $f_1$ and $f_4$ are given arbitrarily this is in standard form so the local solution exists. Then we see that we can choose $\omega=dx^1$, and $a$ and $b$ are given by
\begin{align*}
    a=& \frac{f_1f_2f_3(f_4')^2+f_2f_3f_4f_1'f_4'+f_1f_2f_4f_3'f_4'-f_1f_3f_4f_2'f_4'-2f_1f_2f_3f_4f_4''}{4f_1^2f_2f_3f_4^2}    \\
   b=&\frac{f_1f_2f_3(f_4')^2+f_2f_3f_4f_1'f_4'+f_1f_4^2f_2'f_3'-2f_1f_2f_3f_4f_4''}{2f_1^2f_2f_3f_4^2}   
\end{align*}
 \label{ex-qe1-4d2}
\end{example}

\begin{example}
Problem \ref{prob-qe2} with metric \eqref{metric-4d2}.

Here we choose that $\lambda$ also is only function of $x^1$. This gives the system
\begin{align*}
    f_2''=&\frac{f_1f_3f_4(f_2')^2+f_2f_3f_4f_1'f_2'
    -f_1f_2f_4f_2'f_3'-f_1f_2f_3f_2'f_4'-4f_1f_2f_3f_4f_2'\lambda'}{2f_1f_2f_3f_4}\\
    f_3''=&\frac{f_1f_2f_4(f_3')^2+f_2f_3f_4f_1'f_3'
    -f_1f_3f_4f_2'f_3'-f_1f_2f_3f_3'f_4'-4f_1f_2f_3f_4f_3'\lambda'}{2f_1f_2f_3f_4}\\
    f_4''=&\frac{f_1f_2f_3(f_4')^2+f_2f_3f_4f_1'f_4'
    -f_1f_3f_4f_2'f_4'-f_1f_2f_4f_3'f_4'
    -4f_1f_2f_3f_4f_4'\lambda'}{2f_1f_2f_3f_4}\\
    \lambda''=&\frac{f_1(f_2f_3'f_4'+f_2'f_3f_4'+f_2'f_3'f_4)
    +2\lambda'\big(f_1'f_2f_3f_4+f_1f_2'f_3f_4
    +f_1f_2f_3'f_4+f_1f_2f_3f_4'\big)}{4f_1f_2f_3f_4}
\end{align*}
This looks complicated but actually we can solve it in terms of $f_2$. So denoting $f=f_2$ and $\nu=\sqrt{1+n^2+m^2}$ we obtain
\begin{align*}
    f_1=&c_1c_2c_3f^{-\nu-2}(f')^2\\
    f_3=&c_2 f^m\\
    f_4=&c_3 f^n\\
    \lambda=&-\frac{1+n+m+\nu}{4}\,\ln(f)
\end{align*}

 \label{ex-qe2-4d2}
 \end{example}

\subsection{Third 4 dimensional case}
Let us consider the following four dimensional case:
\begin{equation}
  g=q(x^4)(dx^{1})^{2}+u(x^3) (dx^{2})^{2}
  +h(x^2) (dx^{3})^{2} +f(x^1) (dx^{4})^{2}    
  \label{metric-4d3}
\end{equation}

\begin{example}
Problem \ref{prob-rr} with metric \eqref{metric-4d3}.

In the $\mathsf{RR}$ case we have 16 equations. Again there are sub cases but the most general solution is the following. For $h$ and $u$ we have the equations
\begin{align*}
    h'''&=\frac{h'(2hh''-(h')^2)}{2h^2}\\
    u''&=\frac{u(h')^2+h(u')^2-2huh''}{2hu}
\end{align*}
Again the second equation depends on $h$ only through initial conditions. Hence by standard theorems local solutions exist for these equations. 
Then $\beta$ is of the form $\beta=\beta_1 dx^1+\beta_4 dx^4$ where $\beta_j$ depend only on $f$ and $q$.

\label{ex-rr-4d3}
\end{example}

\begin{example}
Problem \ref{prob-prs} with metric \eqref{metric-4d3}.

The system is incompatible with the requirement $\alpha\ne 0$, so there are no $\mathsf{PRS}$ metrics of this form.
\label{ex-prs-4d3}
\end{example}

\begin{example}
Problem \ref{prob-cc} with metric \eqref{metric-4d3}.

In this case the equations yield only solutions where $\nabla \mathsf{Ri}=0$, so there are no interesting examples of this form.
\label{ex-cc-4d3}
\end{example}

\begin{example}
Problem \ref{prob-qe1} with metric \eqref{metric-4d3}.

Computing the minors and solving for $a$ we obtain a system which requires that two of the functions $f$, $h$, $u$ and $q$ must be constants. This leads to a solution with nontrivial Ricci tensor, but this Ricci tensor is necessarily Einstein so there are no nontrivial quasi Einstein metrics of this form.

\label{ex-qe1-4d3}
\end{example}

\begin{example}
Problem \ref{prob-qe2} with metric \eqref{metric-4d3}.

Here the most general solution to the differential equations is incompatible with positive definiteness of the metric. For the next general system one has $\lambda_2=\lambda_3=f=1$ and for other functions we obtain
\begin{align*}
  &  2h^2 h'''-h'\big(2hh''-(h')^2\big)=0\\
  & 2qq''-(q')^2=0\\
 &2hu u''+2huh''-h(u')^2-u(h')^2=0\\
 &  4q\lambda_1''+(q')^2\lambda_1=0\\
    &   2q\lambda_4'-\lambda_4 q'=0
\end{align*}
This has the following family of solutions
\begin{align*}
    q=&(c_1x^4+c_0)^2\\
     u=&(c_2x^3+c_3)^2\\
      h=&(c_4x^2+c_5)^2\\
         \lambda_4=&c_6(c_1x^4+c_0)\\
           \lambda_1=&c_7\cos(c_1x^1)+c_8\sin(c_1x^1)
\end{align*}
Here $\mathsf{Ri}=\nabla\nabla\lambda=0$, but of course $\widehat{\mathsf{Ri}}$ is nontrivial. 
\label{ex-qe2-4d3}
\end{example}

\section{Conclusion}
We have seen above that with appropriate methods one can readily get large classes of nontrivial examples of various classes of Riemannian manifolds. The particular  initial form of the Riemannian metric was not critical; testing different choices revealed that typically one could always find solutions. The formal theory of PDE can be also useful for exploring other classes of Riemannian manifolds, and in general related questions in differential geometry. Spivak writes in \cite[p. 189]{spivak} that many problems in differential geometry are in fact problems of overdetermined PDE whose solution require that one knows all the integrability conditions. Since manipulating PDE systems by hand is typically extremely tedious, one has invented "incredibly concise and elegant ways to state the integrability conditions [...] without ever even mentioning partial derivatives.'' However, the tools which were used above  to compute the involutive form of various systems can perhaps be helpful in many other contexts in differential geometry,  at least in producing relevant (counter)examples in various situations.

\printbibliography

\end{document}